\documentclass[12pt,a4paper]{article}
\usepackage{amsmath}  
\usepackage{amssymb}  
\usepackage{mathrsfs} 
\usepackage{graphicx}
\usepackage{color}
\usepackage{epstopdf}
\usepackage[english]{babel}
\usepackage{amsthm}
\usepackage{cite}
\usepackage{tikz}
\usepackage{enumitem}
\tikzset{->-/.style={decoration={
			markings,
			mark=at position #1 with {\arrow{>}}},postaction={decorate}}}
\usetikzlibrary{shapes,backgrounds,calc}
\usepackage{graphics}

\usepackage{amssymb, mathrsfs, amsfonts}
\newtheorem{theorem}{Theorem}
\newtheorem{cor}{Corollary}
\newtheorem{lem}{Lemma}

\newtheorem{remark}{Remark}
\newtheorem{ex}{Example}

\usepackage{caption}
\usepackage{subcaption}
\usepackage{setspace}

\title{New improved lower bounds for Zagreb indices of graphs}
\author{
	{Mamta Verma$^a$, Ravinder Kumar$^a$  }\\
	\\
	$^a${\small
		Department of Mathematics and Computing,}\\
	{\small Dr B. R. Ambedkar National Institute of Technology, Jalandhar, India.}\\
	{\small E-mail:
		mamtav.ma.22@nitj.ac.in; mamtavermag98765@gmail.com}\\
	{ \small thakurrk@nitj.ac.in; ravithakur557@gmail.com.}\\
}
\date{\quad}
\begin{document}
	\maketitle
	\begin{abstract}
		This paper presents new lower bounds for the first general Zagreb index $Z_{\alpha}(G)$ involving two, three, and four arbitrary degrees of vertices of a simple graph $G$. For the special cases $\alpha = 2$ and $\alpha = -2$, the results give sharper bounds for the first Zagreb index $M_1(G)$ and the modified first Zagreb index ${}^{m}M_1(G)$, thereby improving several well-known inequalities in the literature. Furthermore, some applications of the derived bounds for $M_1(G)$ are demonstrated, establishing new bounds for the second Zagreb index, the spectral radius, Nordhaus-Gaddum type bounds, and their corresponding coindices.
	\end{abstract}
	\noindent
	\textbf{Keywords}: Simple graph, degree-based topological index, first general Zagreb index, modified first Zagreb index, coindices, arithmetic-harmonic mean inequality.
	\noindent\\
	\textbf{Mathematics Subject Classification}:  05C07, 05C09, 05C50. \\
	
	
	\section{Introduction} \label{s1}
	Let $G = (V, E)$ be a finite, simple graph, where $V=V(G) = \{v_1, v_2, \ldots, v_n\}$ denotes the set of vertices and $E=E(G) = \{e_1, e_2, \ldots, e_m\}$ denotes the set of edges. The number of vertices is called the \emph{order} of the graph and is denoted by $n = |V(G)|$, while the number of edges is called the \emph{size} of the graph and is denoted by $m = |E(G)|$. The vertices of $G$ are ordered so that $\Delta = d_1 \geq d_2 \geq \cdots \geq d_{n-1} \geq d_n = \delta$, where $d_i$ is the degree of the vertex $v_i$ in $G$, for $i = 1, 2, \ldots, n$. A graph $G$ is called \emph{regular} if $\Delta = d_1 = d_2 = \cdots = d_{n-1} = d_n = \delta$. Let $\Gamma_{i,j}$ be the class of graphs such that $d_i = d_{i+1} = \cdots = d_{j-1} = d_j$, where $1 \leq i < j \leq n$. Therefore, a graph $G$ is regular if and only if $G \in \Gamma_{1,n}$.
	We use the notation $i \sim j$ to indicate that $v_i$ and $v_j$ are adjacent in $G$, and $i \nsim j$ to indicate that $v_i$ and $v_j$  are not adjacent.
	
	The degree sequence provides essential information about the structure of the graph and serves as the foundation for many degree-based topological indices, including the Zagreb indices. Among the most prominent of these indices was the \emph{first Zagreb index} \cite{gutman1972,gutman1975graph}, defined as
	\begin{equation*}
		M_1(G)=\sum_{i=1}^{n}  d_i^2 = \sum_{i \sim j} \left( d_i + d_j\right).
		\label{eq:zagreb-alpha-2}
	\end{equation*}
	Another closely related and equally significant degree-based index was the \emph{second Zagreb index}, also introduced in the same foundational work \cite{gutman1972}, and defined as
	\begin{equation*}
		M_2(G) = \sum_{i \sim j} d_i d_j.
		\label{eq:second-zagreb}
	\end{equation*}
	In addition, the \emph{forgotten index}, or \emph{F-index} \cite{furtula2015forgotten}, originally mentioned in early studies and later formalized in various works, was defined as
	\begin{equation*}
		F(G) =\sum_{i=1}^{n}  d_i^3.
		\label{eq:f-index}
	\end{equation*}
	Another variant, known as the \emph{modified first Zagreb index} \cite{modzagreb}, was defined as
	\begin{equation*}
		{}^{m}M_1(G) =\sum_{i=1}^{n}  \frac{1}{d_i^2}.
		\label{eq:modified-zagreb}
	\end{equation*}
	The \emph{inverse-degree index} \cite{mukwembi2010diameter} for a graph with no isolated vertices was defined as
	\begin{equation*}
		ID(G) =\sum_{i=1}^{n}  \frac{1}{d_i}.
		\label{eq:inverse-degree}
	\end{equation*}Let $\mathbb{R}$ denote the set of all real numbers.
	A natural generalization of the first Zagreb index was the \emph{generalized first Zagreb index} \cite{hu2005molecular,milicevic2004variable}, also known as the \emph{zeroth-order general Randić index}, defined for any $\alpha \in \mathbb{R}$ as
	\begin{equation*}
		Z_{\alpha}(G) =\sum_{i=1}^{n}  d_i^\alpha =\sum_{i \sim j} \left( d_i^{\alpha - 1} + d_j^{\alpha - 1} \right).
		\label{eq:general-zagreb}
	\end{equation*}
	Note that $Z_{0}(G) = n,Z_{1}(G) = 2m$, $Z_{2}(G) = M_1(G)$, $Z_{3}(G) = F(G)$, $Z_{-1}(G) = ID(G)$, and $Z_{-2}(G) = {}^{m}M_1(G)$.\\
	The concept of a coindex was first introduced in \cite{doslic2008vertex}. The general Zagreb coindex \cite{mansour2012and} was defined as
	\begin{equation*}
		\overline{Z}_{\alpha}(G) = \sum_{i \not \sim j}\left( d_i^{\alpha - 1} + d_j^{\alpha - 1} \right),
		\label{eq:general_randic_coindex}
	\end{equation*}
	where $\alpha$ is an arbitrary real number.
	In \cite{milovanovic2020note}, the authors proved that if $\alpha$ is any real number, then
	\begin{equation} \label{eq:randic-sum}
		Z_{\alpha}(G) + \overline{Z}_{\alpha}(G) = (n - 1) Z_{\alpha - 1}(G).
	\end{equation}
	Similarly, the \emph{second Zagreb coindex} of a simple graph $G$ was also defined in \cite{gutman2015zagreb} as
	\begin{equation*}
		\overline{M}_2(G) = \sum_{i \not \sim j} d_i  d_j.
		\label{eq:second-zagreb-coindex}
	\end{equation*}
	For more information about Zagreb indices and their mathematical properties; see \cite{ahmad2025sombor,borovicanin2016extremal,de1998upper,das2003sharp,das2016comparison,yang2022maximum,yoon2006relationship} and references therein.\\The problem of determining sharp lower bounds for the first Zagreb index $M_1(G)$
	has attracted considerable attention in recent years.
	Several results are already known in the literature; see\cite{furtula2015forgotten,de1998upper,das2003sharp,das2015zagreb,mansour2012and,yoon2006relationship}. In the present context, we need some lower bounds for $M_1(G)$. For this, consider $ G $ be a simple graph with $ n \geq 3 $ vertices and $ m $ edges. Let $\Delta,d_{2},d_{n-1},$ and $\delta$ be the largest, second largest, second smallest, and smallest  degree, respectively.  Some lower bounds for $M_1(G)$ are listed as follows:
	\begin{align}
		M_1(G) &\geq \frac{4m^2}{n} + \frac{(\Delta - \delta)^2}{2},
		\label{eq:xu2} \\
		M_1(G) &\geq \Delta^2 + d_{2}^2 + \frac{(2m - \Delta - d_{2})^2}{n - 2},
		\label{eq:xu1} \\
		M_1(G) &\geq \Delta^2 + \frac{(2m - \Delta)^2}{n - 1} + \frac{(d_{2} - \delta)^2}{2},
		\label{eq:avg-corollary} \\
		M_1(G) &\geq \delta^2 + \frac{(2m - \delta)^2}{n - 1} + \frac{(\Delta - d_{n-1})^2}{2}
		\label{eq:avg-corollary1}
	\end{align}and
	\begin{equation}
		\quad  \hspace{1.8cm}M_1(G) \geq \Delta^2 + \delta^2 + \frac{(2m - \Delta - \delta)^2}{n - 2} + \frac{(d_{2} - d_{n-1})^2}{2}.
		\label{eq:xu3}
	\end{equation}
	The inequality \eqref{eq:xu2} was proved in \cite{das2015zagreb}, \eqref{eq:xu1} in \cite{mansour2016new}, and \eqref{eq:avg-corollary}, \eqref{eq:avg-corollary1} and \eqref{eq:xu3} in \cite{gutman2017generalizations}.
	
	The paper is organized as follows. Section \ref{s1} serves as the introduction to the paper,  including the basic definitions and terminology used throughout the paper. In Section \ref{s2}, we begin by establishing a key lemma that provide bounds for the sum of the squares of $n$ real numbers in terms of their mean and any two arbitrary numbers. This lemma serves as the foundational tool for deriving several new results. In particular using Lemma 1, Theorems \ref{zt1}, \ref{zt2}, and \ref{zt3}  provide new lower bounds for $ Z_{\alpha}(G)$ involving $2,3$ and $4$  arbitrary vertices, respectively. For specific values of the parameter $\alpha$, especially $\alpha = 2$, Corollaries \ref{p1}--\ref{c35} give new lower bounds for the first Zagreb index $M_1(G)$, thereby improving several existing inequalities in the literature, and also show in Example \ref{ex1} that our bounds are independent of each other.  
	Moreover, we extend the application of our  theorems to derive lower bounds for other important degree-based topological indices, including the second Zagreb index, the spectral radius, Nordhaus-Gaddum type bounds, and their corresponding coindices (see; Remarks \ref{zr5}--\ref{zr4}). Furthermore, using Theorems \ref{zt1}--\ref{zt3}, we also present improved lower bounds for the modified first Zagreb index in Corollaries \ref{pm4}--\ref{pm7}. Overall, in this work we improved several bounds that exist in the literature; see \cite{das2003sharp,das2013zagreb,das2015zagreb,mansour2016new,gutman2017generalizations,gutman2015zagreb,milovanovic2015sharp,milovanovic2020note,milovsevic2019some,li2019second,matejic2022bounds,milovanovic2023some}.
	\section{Main Results}\label{s2}
	We begin our work by presenting the following lemma, which is helpful for proving our main Theorems~\ref{zt1}--\ref{zt3}.
	\begin{lem}\label{zl2}
		Let $x_1, x_2, \dots, x_n$ be $n ~(\geq 3)$ real numbers with mean $\bar{x}$. Then, for any indices $j, k \in \{1, 2, \dots, n\}$ with $j \ne k$, we have
		\begin{align}\label{zlle1}
			\sum_{i=1}^n x_i^2
			\geq     n \bar{x}^2
			+  \frac{1}{2}(x_j - x_k)^2
			+ \frac{2n}{n - 2} \left( \bar{x} - \frac{x_j + x_k}{2} \right)^2.
		\end{align}
		Equality holds in \eqref{zlle1} if and only if
		\begin{equation}\label{equality}
			x_1 = x_2 = \dots = x_{j-1} = x_{j+1} = \dots = x_{k-1} = x_{k+1} = \dots = x_n,
			~ \text{with} ~ j \neq k.
		\end{equation}
	\end{lem}
	\begin{proof} Let $\tau = \{1, 2, \dots, n\} \setminus \{j, k\}$ and define the set $T = \{ x_i \mid i \in \tau \}$  with mean $\bar{x}_{\tau}$.
		We observe that\begin{align}
			\sum_{i=1}^n (x_i - \bar{x})^2 &=   (x_j - \bar{x})^2 + (x_k - \bar{x})^2+\sum_{i \in \tau} \left( x_i - \bar{x} \right)^2\notag\\
			&= (x_j - \bar{x})^2 + (x_k - \bar{x})^2+\sum_{i \in \tau} \left( x_i -\bar{x}_{\tau}+\bar{x}_{\tau}- \bar{x} \right)^2 \nonumber\\
			&= (x_j - \bar{x})^2 + (x_k - \bar{x})^2+\sum_{i \in \tau} (x_i - \bar{x}_{\tau})^2 + (n - 2)(\bar{x}_{\tau} - \bar{x})^2+ (\bar{x}_{\tau} - \bar{x})\sum_{i \in \tau} (x_i - \bar{x}_{\tau}).
			\label{zle2}
		\end{align}Note that\begin{equation}\label{zle33}
			\sum_{i \in \tau} (x_i - \bar{x}_{\tau})=\sum_{i \in \tau} \left( x_i - \frac{1}{n - 2} \sum_{i \in \tau} x_i\right) =\sum_{i \in \tau}  x_i -\sum_{i \in \tau} x_i=0.
		\end{equation}Also, we have
		\begin{align}
			(\bar{x}_{\tau} - \bar{x})^2
			&= \left( \frac{\sum_{i \in \tau} x_i}{n - 2} - \frac{\sum_{i=1}^n x_i}{n} \right)^2
			= \left( \frac{1}{n(n - 2)} \left(  -n x_j - n x_k + 2 \sum_{i=1}^n x_i\right)  \right)^2 \notag \\
			&= \frac{4}{(n - 2)^2} \left(\bar{x}- \frac{x_j + x_k}{2}   \right)^2.
			\label{zle8}
		\end{align}
		Similarly,
		\begin{align}
			(x_j - \bar{x})^2 + (x_k - \bar{x})^2 &=
			\left( x_j - \frac{x_j + x_k}{2} + \frac{x_j + x_k}{2} - \bar{x} \right)^2
			+ \left( x_k - \frac{x_j + x_k}{2} + \frac{x_j + x_k}{2} - \bar{x} \right)^2 \notag\\
			&= \frac{1}{2}(x_j - x_k)^2 + 2 \left( \frac{x_j + x_k}{2} - \bar{x} \right)^2.
			\label{zle10}
		\end{align}
		Combining \eqref{zle2}--\eqref{zle10}, we get
		\begin{equation}
			\sum_{i=1}^n (x_i - \bar{x})^2 =  \frac{1}{2}(x_j - x_k)^2 + \frac{2n}{n - 2} \left( \bar{x} - \frac{x_j + x_k}{2} \right)^2+\sum_{i \in \tau} (x_i - \bar{x}_{\tau})^2.
			\label{zle1}
		\end{equation}
		Also, it is easy to see that
		\begin{align}\label{zle4}
			\sum_{i=1}^n (x_i - \bar{x})^2 = \sum_{i=1}^n x_i^2 - n \bar{x}^2 \quad \text{and} \quad \sum_{i \in \tau} (x_i - \bar{x}_{\tau})^2 \geq 0.
		\end{align}
		Combining \eqref{zle1} and \eqref{zle4}, we immediately get \eqref{zlle1}.\\
		Equality holds in \eqref{zlle1} if and only if $\sum_{i \in \tau} (x_i - \bar{x}_{\tau})^2 =0,$ which is possible  if and only if all $x_i \in T$ are equal.
	\end{proof}
	We now present a lower bound for $Z_{2\alpha}(G)$ in terms of any two arbitrary vertex degrees of a simple graph $G$.
	\begin{theorem}\label{zt1}
		Let $G$ be a simple graph with $n \geq 3$ vertices and $m$ edges. Let $ \Delta =d_1 \geq d_2 \geq \cdots \geq d_{n-1}\geq d_n=\delta>0$ be the degree sequence. Then for any $\alpha \in \mathbb{R}$,
		\begin{equation}\label{eq:randic_general}
			Z_{2\alpha}(G) \geq \frac{\left( Z_{\alpha}(G) \right)^2}{n} + \frac{1}{2} \left( d_j^{\alpha} - d_k^{\alpha} \right)^2 + \frac{2n}{n - 2} \left(\frac{Z_{\alpha}(G)}{n} - \frac{d_j^{\alpha} + d_k^{\alpha}}{2} \right)^2,
		\end{equation}
		with distinct indices $j \ne k$, where $j, k \in \{1, 2, \ldots, n\}$.
	\end{theorem}
	\begin{proof} Let  $x_i = d_i^{\alpha},$ $i = 1, 2, \ldots, n$, and $\alpha \in \mathbb{R}.$ Then, we can write
		\begin{equation}\label{rze1}
			\bar{x} = \frac{1}{n} \sum_{i=1}^{n} d_i^{\alpha} = \frac{Z_{\alpha}(G)}{n}, \quad \text{and} \quad\sum_{i=1}^{n} x_i^2 = \sum_{i=1}^{n} d_i^{2\alpha}={Z_{2\alpha}(G)}.
		\end{equation}
		Applying Lemma \ref{zl2} for $x_i = d_i^{\alpha},$ we have
		\begin{equation}\label{rze2}
			\sum_{i=1}^{n} d_i^{2\alpha} \geq n\bar{x}^2 + \frac{1}{2} (d_j^{\alpha} - d_k^{\alpha})^2 + \frac{2n}{n - 2} \left( \bar{x} - \frac{d_j^{\alpha} + d_k^{\alpha}}{2} \right)^2.
		\end{equation}
		The inequality \eqref{eq:randic_general}, now follows immediately on combining \eqref{rze1} and \eqref{rze2}.
	\end{proof}
	\begin{remark} In \cite{milovsevic2019some}, the following inequality was proved:
		\begin{equation}  \label{eq:randic_diff_main}
			n \cdot Z_{2\alpha}(G) - \left( Z_{\alpha}(G) \right)^2 \geq \frac{n}{2} \left( \Delta^{\alpha} - \delta^{\alpha} \right)^2.
		\end{equation}
		Clearly, our inequality \eqref{eq:randic_general} is stronger than the inequality \eqref{eq:randic_diff_main} by choosing $d_j=\Delta$ and $d_k=\delta$ in Theorem \ref{zt1}.
	\end{remark}We now present a consequence of Theorem \ref{zt1} in terms of the first Zagreb index $M_1(G)$.
	\begin{cor}\label{p1}
		Let $G$ be  a simple graph with $n \geq 3$ vertices and $m$ edges.  Let $ \Delta =d_1 \geq d_2 \geq \cdots \geq d_{n-1}\geq d_n=\delta>0$ be the degree sequence. Then
		\begin{equation}\label{zemte1}\nonumber
			M_1(G) \geq \frac{4m^2}{n} + \frac{1}{2}(d_j - d_k)^2 + \frac{2n}{n - 2} \left( \frac{2m}{n} - \frac{d_j + d_k}{2} \right)^2,
		\end{equation}
		for any distinct indices $j \ne k$, where $j, k \in \{1, 2, \ldots, n\}$.
	\end{cor}
	\begin{proof}
		Note that for  $\alpha=1$, we have
		$Z_{\alpha}(G) =  \sum_{i=1}^{n} d_i = {2m}$,
		and $Z_{2\alpha}(G) = \sum d_i^2 = M_1(G)$. Hence, Corollary \ref{p1} follows directly from Theorem \ref{zt1}.
	\end{proof} Corollary~\ref{p1} provides various lower bounds for $M_1(G)$ depending on the choices of two distinct vertex degrees. In particular, by choosing the extreme degrees, that is, the maximum $\Delta$ and the minimum $\delta$, we obtain the following corollary, which improves upon \eqref{eq:xu2}.
	\begin{cor}
		Let $ G $ be a simple graph with $ n\geq3$ vertices and $ m $ edges. Then
		\begin{equation}
			M_1(G) \geq \frac{4m^2}{n} + \frac{1}{2}(\Delta - \delta)^2
			+ \frac{2n}{n-2} \left( \frac{2m}{n} - \frac{\Delta + \delta}{2} \right)^2.
			\label{zte2}
		\end{equation}
		Equality holds in \eqref{zte2} if and only if $G$ is regular or $G \in \Gamma_{2,n-1} $ or $G \in \Gamma_{1,n-1} $ or $G \in \Gamma_{2,n} .$
	\end{cor}
	\begin{proof}
		Inequality \eqref{zte2} follows from Corollary \ref{p1} by taking $d_j=\Delta$ and $d_k=\delta$.\\
		For the equality case, let $G \in \Gamma_{2, n-1}$, which means that $d_2 = d_3 = \dots = d_{n-1} = h $ \text{(say)}.
		Therefore, from \text{
			L.H.S. of \eqref{zte2}}, $M_1(G) = (n-2)h^2 + \Delta^2 + \delta^2 ~ \text{and} ~2m = (n-2)h + \Delta + \delta.$
		Hence,
		\begin{align*}
			\text{R.H.S. of \eqref{zte2}}
			&= \frac{\left( (n-2)h + \Delta + \delta \right)^2}{n}
			+ \frac12 (\Delta - \delta)^2
			+ \frac{2n}{n-2} \left( \frac{ (n-2)h + \Delta + \delta}{n} - \frac{\Delta + \delta}{2} \right)^2 \\[6pt]
			&= \frac{\left( (n-2)h + \Delta + \delta \right)^2}{n}
			+ \frac12 (\Delta - \delta)^2
			+ \frac{n-2}{2n} \left( 2h - \Delta + \delta \right)^2 \\[6pt]
			&= (n-2)h^2 + \Delta^2 + \delta^2.
		\end{align*}
		Likewise, if $G$ is regular or $G \in \Gamma_{2,n} $ or $G \in \Gamma_{1,n-1} $, then equality holds.  \\
		Conversely, if equality in \eqref{zte2} holds, then from \eqref{zle2},  $d_2 = d_3 = \dots = d_{n-1} $, that is,
		if $G$ is regular or $G \in \Gamma_{2,n-1} $ or $G \in \Gamma_{1,n-1} $ or $G \in \Gamma_{2,n} .$
	\end{proof}
	Likewise, the following bounds are immediate consequences of Corollary \ref{p1}.
	\begin{cor}\label{zr2}
		Let $ G $ be a simple graph with $ n\geq 3$ vertices and $ m $ edges. Let $\Delta,d_{2},d_{n-1}$ and $\delta$ be the largest, second largest,  second smallest, and smallest degree, respectively. Then
		\begin{align}
			M_1(G) \geq \frac{4m^2}{n} + \frac{1}{2}(d_{n-1} - \delta)^2
			+ \frac{2n}{n-2} \left( \frac{2m}{n} - \frac{d_{n-1} + \delta}{2} \right)^2
			\label{z2te1}
		\end{align}
		and
		\begin{equation}
			M_1(G) \geq \frac{4m^2}{n} + \frac{1}{2}(\Delta - d_{2})^2
			+ \frac{2n}{n-2} \left( \frac{2m}{n} - \frac{\Delta + d_{2}}{2} \right)^2.
			\label{z2te2}
		\end{equation}
		Equality holds in \eqref{z2te1} if and only if $G$ is regular or $G \in \Gamma_{1,n-1} $ or $G \in \Gamma_{1,n-2}.$\
		Equality holds in \eqref{z2te2} if and only if $G$ is regular or $G \in \Gamma_{3,n} $ or $G \in \Gamma_{2,n}.$
	\end{cor}
	We now show by means of example that our bounds \eqref{zte2}, \eqref{z2te1}, and \eqref{z2te2} are mutually incomparable.
	\begin{ex}\label{ex1}
		Let $G_1,G_2$ and $G_3$ be three simple graphs as given below:
		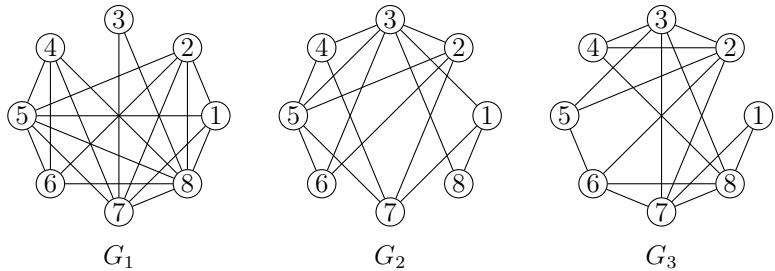
\begin{figure}[h!]
			\centering
			
			\begin{tikzpicture}[scale=0.85, every node/.style={scale=0.85}]
				
				\begin{scope}[shift={(0,0)}]
					\def\n{8}
					\def\r{1.5}
					\foreach \i in {1,...,\n} {
						\node[circle, draw, fill=white, inner sep=1pt] (G1N\i) at ({360/\n * (\i - 1)}:\r) {\i};
					}
					
					\draw (G1N1)--(G1N2);
					\draw (G1N1)--(G1N5);
					\draw (G1N1)--(G1N7);
					\draw (G1N1)--(G1N8);
					
					\draw (G1N2)--(G1N5);
					\draw (G1N2)--(G1N6);
					\draw (G1N2)--(G1N7);
					\draw (G1N2)--(G1N8);
					
					\draw (G1N3)--(G1N7);
					\draw (G1N3)--(G1N8);
					
					\draw (G1N4)--(G1N5);
					\draw (G1N4)--(G1N6);
					\draw (G1N4)--(G1N7);
					\draw (G1N4)--(G1N8);
					
					\draw (G1N5)--(G1N6);
					\draw (G1N5)--(G1N7);
					\draw (G1N5)--(G1N8);
					
					\draw (G1N6)--(G1N8);
					
					\draw (G1N7)--(G1N8);
					\node at (0,-2.2) {$G_1$};
				\end{scope}
				
				\begin{scope}[shift={(4.2,0)}]
					\def\n{8}
					\def\r{1.5}
					\foreach \i in {1,...,\n} {
						\node[circle, draw, fill=white, inner sep=1pt] (G2N\i) at ({360/\n * (\i - 1)}:\r) {\i};
					}
					\draw (G2N1)--(G2N3); \draw (G2N1)--(G2N7); \draw (G2N1)--(G2N8);
					\draw (G2N2)--(G2N3); \draw (G2N2)--(G2N5); \draw (G2N2)--(G2N6); \draw (G2N2)--(G2N7);
					\draw (G2N3)--(G2N4); \draw (G2N3)--(G2N5); \draw (G2N3)--(G2N6); \draw (G2N3)--(G2N8);
					\draw (G2N4)--(G2N5); \draw (G2N4)--(G2N7);
					\draw (G2N5)--(G2N6); \draw (G2N5)--(G2N7);
					\node at (0,-2.2) {$G_2$};
				\end{scope}
				
				\begin{scope}[shift={(8.4,0)}]
					\def\n{8}
					\def\r{1.5}
					\foreach \i in {1,...,\n} {
						\node[circle, draw, fill=white, inner sep=1pt] (G3N\i) at ({360/\n * (\i - 1)}:\r) {\i};
					}
					\draw (G3N1)--(G3N7); \draw (G3N1)--(G3N8);
					\draw (G3N2)--(G3N3); \draw (G3N2)--(G3N4); \draw (G3N2)--(G3N5); \draw (G3N2)--(G3N6); \draw (G3N2)--(G3N7);
					\draw (G3N3)--(G3N4); \draw (G3N3)--(G3N5); \draw (G3N3)--(G3N7); \draw (G3N3)--(G3N8);
					\draw (G3N4)--(G3N8);
					\draw (G3N5)--(G3N6);
					\draw (G3N6)--(G3N7); \draw (G3N6)--(G3N8);
					\draw (G3N7)--(G3N8);
					\node at (0,-2.2) {$G_3$};
				\end{scope}
				
			\end{tikzpicture}
			
			\caption{Simple graphs.}
		\end{figure}
		\newpage
		\begin{table}[h!]
			\centering
			\renewcommand{\arraystretch}{1.4}
			\begin{tabular}{|c|c|c|c|c|c|c|c|c|}
				\hline
				\textbf{Graph} & $\Delta$ & $d_{2}$ & $d_{n-1}$ & $\delta$ & $M_1(G)$ & \eqref{zte2} & \eqref{z2te1} & \eqref{z2te2} \\
				\hline
				$G_1$ & 7 & 6 & 4 & 2 & 198.00 & 193.1667 & 189.1667 & 190.6667 \\
				\hline
				$G_2$ & 6 & 5 & 3 & 2 & 124.00 & 120.6667 & 121.1667 & 117.1667 \\
				\hline
				$G_3$ & 5 & 4 & 3 & 2 & 138.00 & 133.1667 & 129.1667 & 134.5000 \\
				\hline
			\end{tabular}
			\caption{Bounds based on degree pairs for the first Zagreb index.}
		\end{table}
	\end{ex}
	Next, we present a lower bound for $Z_{2\alpha}(G)$ in terms of the degrees of three arbitrary vertices of a graph $G$.
	\begin{theorem}\label{zt2}
		Let $G$ be a simple graph with $n\ge4$ vertices  and $m$ edges. Let $ \Delta =d_1 \geq d_2 \geq \cdots \geq d_{n-1}\geq d_n=\delta>0$ be the degree sequence. Then for any $\alpha \in \mathbb{R}$
		{ \begin{equation}\label{eq:general_randic_bound}
				Z_{2\alpha}(G) \geq d_\ell^{2\alpha} + \frac{(Z_{\alpha}(G) - d_\ell^{\alpha})^2}{n - 1} + \frac{1}{2}\left( d_j^{\alpha} - d_k^{\alpha} \right)^2 + \frac{2(n - 1)}{n - 3} \left( \frac{Z_{\alpha}(G)-d_\ell^{\alpha}}{n - 1} - \frac{d_j^{\alpha} + d_k^{\alpha}}{2} \right)^2,
		\end{equation}} where $j, k, \ell \in \{1, 2, \dots, n\}$ and $j \neq k \neq \ell$.
	\end{theorem}
	\begin{proof}
		Let $ T = \{ d_i^{\alpha} : i = 1, 2, \dots, n,\ i \neq \ell \} $ be a set of $ n - 1 $ real numbers, with mean $ \overline{T} $. Then
		\begin{equation}\label{eq:randic_t_mean}
			\overline{T} = \frac{1}{n - 1} \sum_{\substack{i=1 \\ i \neq \ell}}^{n} d_i^{\alpha} = \frac{Z_{\alpha}(G) - d_\ell^{\alpha}}{n - 1}, \quad \text{and} \quad \sum_{\substack{i=1 \\ i \neq \ell}}^{n} d_i^{2\alpha} = Z_{2\alpha}(G) - d_\ell^{2\alpha}.
		\end{equation}
		Applying Lemma \ref{zl2} to the set $T$ with $x_j = d_j^{\alpha}$ and $x_k = d_k^{\alpha}$, we find that
		\begin{equation}\label{eq:randic_refined_inequality}
			\sum_{\substack{i=1 \\ i \neq \ell}}^{n} d_i^{2\alpha} \geq (n - 1)\overline{T}^2 + \frac{1}{2}(d_j^{\alpha} - d_k^{\alpha})^2 + \frac{2(n - 1)}{n - 3} \left( \overline{T} - \frac{d_j^{\alpha} + d_k^{\alpha}}{2} \right)^2.
		\end{equation}
		Combining \eqref{eq:randic_t_mean} and \eqref{eq:randic_refined_inequality}, we get \eqref{eq:general_randic_bound}.
	\end{proof}The following corollary is an immediate consequence of Theorem \ref{zt2}.
	\begin{cor}\label{p2}
		Let $G$ be  a simple graph with $n \geq 4$ vertices and $m$ edges.  Let $ \Delta =d_1 \geq d_2 \geq \cdots \geq d_{n-1}\geq d_n=\delta>0$ be the degree sequence. Then
		\begin{equation}\label{zemte3}
			M_1(G) \geq d_\ell^2 + \frac{(2m - d_\ell)^2}{n - 1} + \frac{1}{2}(d_j - d_k)^2 + \frac{2(n-1)}{n - 3} \left( \frac{2m - d_\ell}{n - 1} - \frac{d_j + d_k}{2} \right)^2,
		\end{equation}
		where $j, k, \ell \in \{1, 2, \dots, n\}$ are distinct indices.
	\end{cor}\begin{proof}
		The proof of the corollary follows immediately by selecting $\alpha=1$ in Theorem \ref{zt2}.
	\end{proof}
	Note that the inequality \eqref{zemte3} is stronger than the inequality  
	\begin{equation*}
		M_1(G) \geq d_\ell^2 + \frac{(2m - d_\ell)^2}{n - 1} + \frac{1}{2}(\Delta -\delta)^2,
		\label{zt3e2}
	\end{equation*}
	which was   given in \cite{gutman2017generalizations} by taking $d_j=\Delta, d_k=\delta$.\\
	By specializing Corollary \ref{p2} in such a manner that they correspond to the largest, second largest, second smallest, and smallest degrees of $G$, we obtain the following corollary.
	
	\begin{cor}\label{c3}
		Let $ G $ be a simple graph with $ n \geq 4 $ vertices and $ m $ edges. Let $\Delta,d_{2},d_{n-1}$ and $\delta$ be the largest, second largest, second smallest, and smallest  degree, respectively. Then
		\begin{align}
			M_1(G) \geq \Delta^2 + \frac{(2m - \Delta)^2}{n - 1}
			+ \frac{1}{2}(d_{2} - \delta)^2
			+ \frac{2(n - 1)}{n - 3} \left( \frac{2m - \Delta}{n - 1} - \frac{d_{2} + \delta}{2} \right)^2
			\label{zr31}
		\end{align}and
		\begin{align}
			M_1(G) \geq \delta^2 + \frac{(2m - \delta)^2}{n - 1}
			+ \frac{1}{2}(\Delta - d_{n-1})^2
			+ \frac{2(n - 1)}{n - 3} \left( \frac{2m - \delta}{n - 1} - \frac{\Delta + d_{n-1}}{2} \right)^2.
			\label{zr32}
		\end{align}
		Equality holds in \eqref{zr31} if and only if $G$ is regular, or $G \in \Gamma_{3,n-1}$, or $G \in \Gamma_{3,n}$, or $G \in \Gamma_{2,n-1}$, or $G \in \Gamma_{2,n}$, or $G \in \Gamma_{1,n-1}$. Equality holds in \eqref{zr32} if and only if $G$ is regular, or $G \in \Gamma_{2,n-2}$, or $G \in \Gamma_{2,n-1}$, or $G \in \Gamma_{2,n}$, or $G \in \Gamma_{1,n-2}$, or $G \in \Gamma_{1,n-1}$.
	\end{cor}
	It is easy to see that the inequalities \eqref{zr31} and \eqref{zr32} are stronger than \eqref{eq:avg-corollary} and \eqref{eq:avg-corollary1}, respectively. \\
	
	We now establish a lower bound for $Z_{2\alpha}(G)$ expressed in terms of the degrees of any four vertices of a simple graph $G$.
	\begin{theorem}\label{zt3}
		Let $G$ be a simple graph with $n\ge5 $ vertices and $m$ edges. Let $ \Delta =d_1 \geq d_2 \geq \cdots \geq d_{n-1}\geq d_n=\delta>0$ be the degree sequence. Then for any $\alpha \in \mathbb{R}$,
		\begin{equation}
			\hspace{-.1cm} Z_{2\alpha}(G) \geq d_l^{\alpha} + d_m^{\alpha} +  \frac{\left(Z_{\alpha}(G) - d_l^{\alpha} - d_m^{\alpha}\right)^2}{n - 2}  + \frac{1}{2}(d_j^{\alpha} - d_k^{\alpha})^2
			+ \frac{2(n - 2)}{n - 4} \left( \frac{Z_{\alpha}(G) - d_l^{\alpha} - d_m^{\alpha}}{n - 2} - \frac{d_j^{\alpha} + d_k^{\alpha}}{2} \right)^2,
			\label{eq:z4te1_randicnew_swapped}
		\end{equation}where $j, k, \ell, m \in \{1,2,\dots,n\}$ are distinct indices.
	\end{theorem}
	\begin{proof}
		Let $ T = \{ d_i^{\alpha} : i = 1, 2, \dots, n,\ i \neq \ell \neq m \} $ be a set of $ n - 2 $ real numbers, with mean $ \overline{T} $. Then
		\begin{equation}
			\overline{T} = \frac{1}{n - 2} \sum_{\substack{i=1\\ i \ne l, m}}^{n} d_i^{\alpha} = \frac{\sum_{i=1}^{n} d_i^{\alpha} - d_l^{\alpha} - d_m^{\alpha}}{n - 2}.
			\label{eq:mean_T_swapped}
		\end{equation}
		Applying Lemma \ref{zl2} for the set $T$ of $n-2$ elements, we get
		\begin{equation}
			\sum_{\substack{i=1\\ i \ne l, m}}^{n} d_i^{2\alpha} \geq (n - 2)\overline{T}^2 + \frac{1}{2}(d_j^{\alpha} - d_k^{\alpha})^2 + \frac{2(n - 2)}{n - 4} \left( \overline{T} - \frac{d_j^{\alpha} + d_k^{\alpha}}{2} \right)^2.
			\label{eq:refined_T_swapped}
		\end{equation}
		Substituting the value of $\overline{T}$ from \eqref{eq:mean_T_swapped} in \eqref{eq:refined_T_swapped} which directly leads to \eqref{eq:z4te1_randicnew_swapped}.
	\end{proof}
	As a direct consequence of Theorem~\ref{zt3}, we obtain the following corollary.
	\begin{cor}\label{p3}
		Let $G$ be  a simple graph with $n \geq 5$ vertices and $m$ edges.  Let $ \Delta =d_1 \geq d_2 \geq \cdots \geq d_{n-1}\geq d_n=\delta>0$ be the degree sequence. Then
		\begin{equation}\label{zemte2}\nonumber
			M_1(G) \geq d_j^{2} + d_k^{2}+ \frac{\left(2m - d_j^{2} - d_k^{2} \right)^2}{n - 2}+ \frac{1}{2} (d_l - d_m)^2
			+ \frac{2\left(n-2 \right) }{n - 4} \left( \frac{2m - d_j - d_k}{n - 2} - \frac{d_l + d_m}{2} \right)^2,
		\end{equation}
		for distinct indices $j, k, l, m \in \{1, 2, \dots, n\}$.
	\end{cor}
	\begin{cor}\label{c35}
		Let $G$ be  a simple graph with $n \geq 5$ vertices and $m$ edges.  Let $ \Delta =d_1 \geq d_2 \geq \cdots \geq d_{n-1}\geq d_n=\delta>0$ be the degree sequence. Then
		\begin{align}
			M_1(G) &\geq \Delta^2 + d_{2}^2 + \frac{(2m - \Delta - d_{2})^2}{n - 2}
			+ \frac{1}{2}(d_{n-1} - \delta)^2 + \frac{2(n - 2)}{n - 4} \left( \frac{2m - \Delta - d_{2}}{n - 2} - \frac{d_{n-1} + \delta}{2} \right)^2
			\label{eq:xu4}
		\end{align}and
		\begin{align}
			M_1(G) &\geq \Delta^2 + \delta^2 + \frac{(2m - \Delta - \delta)^2}{n - 2}
			+ \frac{1}{2}(d_{2} - d_{n-1})^2  + \frac{2(n - 2)}{n - 4} \left( \frac{2m - \Delta - \delta}{n - 2} - \frac{d_{2} + d_{n-1}}{2} \right)^2.
			\label{z24degree}
		\end{align}  Equality holds in \eqref{z24degree} and \eqref{eq:xu4} if and only if $G$ is regular, or $G \in \Gamma_{3,n-2}$, or $G \in \Gamma_{2,n-2}$, or $G \in \Gamma_{3,n-1}$, or $G \in \Gamma_{2,n-1}$, or $G \in \Gamma_{3,n}$, or $G \in \Gamma_{2,n}$, or $G \in \Gamma_{1,n-2}$, or $G \in \Gamma_{1,n-1}$.
	\end{cor}
	Observe that the inequalities \eqref{eq:xu4}  and \eqref{z24degree} are stronger than \eqref{eq:xu1} and \eqref{eq:xu3}, respectively.
	\\Next, we present an inequality involving the first general Zagreb index, the first general Zagreb coindex, and the degrees of vertices of a graph.
	\begin{cor}\label{p4}
		Under the hypothesis of Theorem \ref{zt1}, we have
		\begin{align}\label{rc1}
			Z_{2\alpha+1}(G) +  \overline{Z}_{2\alpha+1}(G)
			\geq (n - 1) \left( \ \frac{{Z_{\alpha}(G)}^2}{n}  + \frac{1}{2} \left( d_j^{\alpha} - d_k^{\alpha} \right)^2   + \frac{2n}{n    - 2} \left( \frac{Z_{\alpha}(G)}{n} - \frac{d_j^{\alpha} + d_k^{\alpha}}{2} \right)^2 \right).
		\end{align}
	\end{cor}
	\begin{proof}
		From the relation between the first general index and coindex in \eqref{eq:randic-sum}, we have  
		\begin{equation} \label{eq:randic-sum1}
			Z_{2\alpha+1}(G) +  \overline{Z}_{2\alpha+1}(G)
			= (n - 1) \, Z_{2\alpha}(G).
		\end{equation}
		Combining  \eqref{eq:randic_general} and \eqref{eq:randic-sum1}, we immediately get \eqref{rc1}.
	\end{proof}Likewise, one can derive bounds for $ Z_{2\alpha+1}(G) +  \overline{Z}_{2\alpha+1}(G)$ using Theorem \ref{zt2} and Theorem \ref{zt3}.\\
	
	We now give some applications of our bounds to derive lower bounds for other important degree-based topological indices, including the second Zagreb index, the spectral radius, Nordhaus-Gaddum type bounds and their corresponding coindices in the following remarks.
	\begin{remark}\label{zr5}
		For the second Zagreb index $M_2(G)$, the following lower bound was given in \cite{das2013zagreb},
		\begin{equation}
			M_2(G) \geq 2m^2 - (n - 1)m\Delta + \frac{1}{2}(\Delta - 1)M_1(G).
			\label{eq:M2_basic_bound}
		\end{equation}
		By substituting various lower bounds for $M_1(G)$ in \eqref{eq:M2_basic_bound}, one can obtain various lower bounds for $M_2(G)$.  
		For instance, the following two bounds are immediate consequences of \eqref{zr31} and \eqref{zr32},
		\begin{equation}\nonumber
			M_2(G) \geq 2m^2 - (n - 1)m\Delta + \frac{1}{2}(\Delta - 1) \Bigg( \delta^2 + \frac{(2m - \delta)^2}{n - 1} + \frac{1}{2}(\Delta - d_{n-1})^2
			+ \frac{2(n - 1)}{n - 3} \Bigg( \frac{2m - \delta}{n - 1} - \frac{\Delta + d_{n-1}}{2} \Bigg)^2 \Bigg)
			\label{eq:M2_bound1}
		\end{equation}and
		\begin{equation}\nonumber
			M_2(G) \geq 2m^2 - (n - 1)m\Delta + \frac{1}{2}(\Delta - 1) \Bigg(  \Delta^2 + \frac{(2m - \Delta)^2}{n - 2} + \frac{1}{2}(d_2 - \delta)^2 +\frac{2(n - 1)}{n - 3} \left( \frac{2m - \Delta}{n - 1} - \frac{d_{2} + \delta}{2} \right)^2\Bigg),
			\label{eq:M2_bound2}
		\end{equation}which are stronger than \begin{equation}\nonumber
			M_2(G) \geq 2m^2 - (n - 1)m\Delta + \frac{1}{2}(\Delta - 1) \Bigg( \delta^2 + \frac{(2m - \delta)^2}{n - 1} + \frac{1}{2}(\Delta - d_{n-1})^2
			\Bigg)
			\label{eq:M2_bound1}
		\end{equation}and
		\begin{equation}\nonumber
			M_2(G) \geq 2m^2 - (n - 1)m\Delta + \frac{1}{2}(\Delta - 1) \Bigg(  \Delta^2 + \frac{(2m - \Delta)^2}{n - 2} + \frac{1}{2}(d_2 - \delta)^2 \Bigg),
			\label{eq:M2_bound2}
		\end{equation} obtained in \cite{gutman2017generalizations}.
	\end{remark}\begin{remark}\label{zr6}
		Denote by $A = (a_{ij})$ the adjacency matrix of a graph $G$. The eigenvalues of $A$ are real and are denoted by $\lambda_1 \geq \lambda_2 \geq \cdots \geq \lambda_n$. The largest eigenvalue $\lambda_1$ is called the spectral radius of $G$.
		It was proved in \cite{hofmeister1988spectral,yu2004spectral} that spectral radius $\lambda_1$ of the graph $G$ satisfies the inequality
		\begin{equation*}
			\lambda_1 \geq \sqrt{\frac{M_1(G)}{n}},
		\end{equation*} and therefore using \eqref{z24degree}, we obtain that
		\begin{equation*}
			\lambda_1 \geq\
			\sqrt{
				\frac{1}{n} \left(
				\Delta^2 + \delta^2
				+ \frac{(2m - \Delta - \delta)^2}{n - 2}
				+ \frac{1}{2}(d_2 - d_{n-1})^2
				+ \frac{2(n - 2)}{n - 4} \left(
				\frac{2m - \Delta - d_{2}}{n - 2}
				- \frac{d_{n-1} + \delta}{2}
				\right)^2 \right)},
		\end{equation*}which is stronger than \begin{equation*}
			\lambda_1 \geq\
			\sqrt{
				\frac{1}{n} \left(
				\Delta^2 + \delta^2
				+ \frac{(2m - \Delta - \delta)^2}{n - 2}
				+ \frac{1}{2}(d_2 - d_{n-1})^2 \right)},
		\end{equation*}obtained in \cite{gutman2017generalizations}.
	\end{remark}
	\begin{remark}\label{zr3}
		Let $ G $ be a simple graph and let  $\overline{G}$ denotes the complement of $G$. Then, the following Nordhaus-Gaddum type relations are valid:
		\begin{align}
			M_1(G) + M_1(\overline{G}) &= n(n-1)^2 - 4m(n-1) + 2 M_1(G),
			\label{eq:M1-complement-sum} \\
			M_2(G) + M_2(\overline{G}) &= \tfrac{n(n-1)^3}{2} + 2m^2 - 3m(n-1)^2
			+ \left(n - \tfrac{3}{2}\right) M_1(G)
			\label{eq:M2-complement-sum}
		\end{align}and
		\begin{align}
			F(G) + F(\overline G) &= n(n - 1)^3 - 6m(n - 1)^2 + 3(n - 1) M_1(G),
			\label{eq:FplusFbar2}
		\end{align}where \eqref{eq:M1-complement-sum} was proved in \cite{das2015zagreb}, \eqref{eq:M2-complement-sum} in \cite{gutman2015zagreb}
		and \eqref{eq:FplusFbar2} in \cite{de2016f}, respectively.
		Hence, using \eqref{zte2} we obtain the following bounds:
		\begin{equation*}\label{eq:M1-complement-sumn}
			\hspace{-.5cm}M_1(G) + M_1(\overline{G}) \geq {n(n-1)^2} - 4m(n-1) + 2 \left( \frac{4m^2}{n} + \frac{1}{2}(\Delta - \delta)^2+ \frac{2n}{n-2} \left( \frac{2m}{n} - \frac{\Delta + \delta}{2} \right)^2\right),
		\end{equation*}
		\begin{equation*}
			\hspace{-.1cm}F(G) + {F}(\overline G) \geq n(n - 1)^3 - 6m(n - 1)^2  + 3(n - 1) \left( \frac{4m^2}{n} + \frac{1}{2}(\Delta - \delta)^2+ \frac{2n}{n-2} \left( \frac{2m}{n} - \frac{\Delta + \delta}{2} \right)^2\right)
		\end{equation*}and
		\begin{equation}\nonumber
			\begin{aligned}\label{eq:M2-complement-sum3}
				\hspace{-2cm}M_2(G) + M_2(\overline{G}) &\geq \frac{n(n-1)^3}{2} + 2m^2 - 3m(n-1)^2  + \left(n - \frac{3}{2}\right)\Bigg( \frac{4m^2}{n} + \frac{1}{2}(\Delta - \delta)^2\\  &\quad+ \frac{2n}{n-2} \left( \frac{2m}{n} - \frac{\Delta + \delta}{2} \right)^2\Bigg).
			\end{aligned}
		\end{equation}which are stronger than \begin{equation*}\label{eq:M1-complement-sumn}
			\hspace{-.5cm}M_1(G) + M_1(\overline{G}) \geq {n(n-1)^2} - 4m(n-1) + 2\left( \frac{4m^2}{n} + \frac{2(n-2)}{(n-1)^2}(\Delta - \delta)^2\right),
		\end{equation*}
		\begin{equation*}
			\hspace{-2cm}F(G) + {F}(\overline G) \geq n(n - 1)^3 - 6m(n - 1)^2  + 3(n - 1) \left( \frac{4m^2}{n} + \frac{1}{2}(\Delta - \delta)^2\right)
		\end{equation*} and
		\begin{equation}\nonumber
			\begin{aligned}\label{eq:M2-complement-sum3}
				\hspace{-2cm}M_2(G) + M_2(\overline{G}) &\geq \frac{n(n-1)^3}{2} + 2m^2 - 3m(n-1)^2  + \left(n - \frac{3}{2}\right)\left( \frac{4m^2}{n} + \frac{2(n-2)}{(n-1)^2}(\Delta - \delta)^2\right),
			\end{aligned}
		\end{equation}
		obtained in \cite{das2015zagreb}, \cite{gutman2017generalizations}, and \cite{das2015zagreb}, respectively.
	\end{remark}
	\begin{remark} \label{zr4}
		In \cite{ashrafi2010zagreb}, it was shown that
		\begin{equation*}\label{eq:M2-bar1}
			\overline{M}_2(G) + M_2(G) = 2m^2 - \frac{1}{2} M_1(G).
		\end{equation*}Therefore using \eqref{zte2}, we get
		\begin{equation*}\label{eq:M2-bar}
			\overline{M}_2(G) + M_2(G) \leq  2m^2 - \frac{1}{2} \left( \frac{4m^2}{n} + \frac{1}{2}(\Delta - \delta)^2+ \frac{2n}{n-2} \left( \frac{2m}{n} - \frac{\Delta + \delta}{2} \right)^2\right),
		\end{equation*} which is stronger than
		\begin{equation*}\label{eq:M2-bar}
			\overline{M}_2(G) + M_2(G) \leq  2m^2 - \frac{1}{2} \left( \frac{4m^2}{n} + \frac{2(n-2)}{(n-1)^2}(\Delta - \delta)^2\right),
		\end{equation*}obtained in \cite{das2015zagreb}.\\
		Likewise, applying \eqref{eq:randic-sum} for $\alpha=2$ and $\alpha=3$, and combining with \eqref{zte2}, we respectively obtain the following bounds:
		\begin{equation*}
			\overline{M}_1(G)  \leq  \frac{2m}{n}\left( n(n-1) - 2m \right) - \frac{1}{2}\left( \Delta - \delta \right)^2- \frac{2n}{n-2} \left( \frac{2m}{n} - \frac{\Delta + \delta}{2} \right)^2
		\end{equation*}and
		\begin{equation*}
			F(G) + \overline{F}(G) \geq (n - 1) \left(\frac{4m^2}{n} + \frac{1}{2}(\Delta - \delta)^2 + \frac{2n}{n-2} \left( \frac{2m}{n} - \frac{\Delta + \delta}{2} \right)^2 \right),
			\label{eq:FplusFbar1}
		\end{equation*}which are stronger than
		\begin{equation*}
			\overline{M}_1(G)  \leq  \frac{2m}{n}\left( n(n-1) - 2m \right) - \frac{1}{2}\left( \Delta - \delta \right)^2
		\end{equation*}and
		\begin{equation*}
			F(G) + \overline{F}(G) \geq (n - 1) \left(\frac{4m^2}{n} + \frac{1}{2}(\Delta - \delta)^2  \right),
			\label{eq:FplusFbar1}
		\end{equation*}obtained in \cite{milovanovic2015sharp} and \cite{gutman2017generalizations}, respectively.
	\end{remark}
	
	Let $G$ be a simple graph with $n \geq 3$ vertices and $m$ edges. Let $ \Delta =d_1 \geq d_2 \geq \cdots \geq d_{n-1}\geq d_n=\delta>0$ be the degree sequence. The following are some lower bounds for the modified first Zagreb index ${}^{m}M_1(G)$,
	\begin{align}
		{}^{m}M_1(G) &\geq \frac{ID(G)^2}{n}
		+ \tfrac{1}{2}\left( \tfrac{1}{\delta} - \tfrac{1}{\Delta} \right)^{2},
		\label{eq:mZagrebBound} \\
		{}^{m}M_1(G) &\geq \frac{1}{\Delta^2}
		+ \frac{(n - 1)^3}{(2m - \Delta)^2},
		\label{eq:ineq_upperbound1} \\
		{}^{m}M_1(G) &\geq \frac{1}{\Delta^2} + \frac{1}{\delta^2}
		+ \sqrt{ \frac{\left(ID(G) - \tfrac{1}{\Delta} - \tfrac{1}{\delta}\right)^3}
			{2m - \Delta - \delta} }
		\label{eq:inverse_degree_bound}
	\end{align}and
	\begin{align}
		{}^{m}M_1(G) &\geq \frac{1}{\Delta^2} + \frac{1}{\delta^2}
		+ \frac{(n - 2)^3}{(2m - \Delta - \delta)^2},
		\label{cor:delta-delta-bound}
	\end{align}where
	\eqref{eq:mZagrebBound} was proved in \cite{gutman2017generalizations}, \eqref{eq:ineq_upperbound1} in \cite{matejic2022bounds}, \eqref{eq:inverse_degree_bound} and \eqref{cor:delta-delta-bound} in \cite{milovanovic2023some}.\\
	
	We next derive some results for the modified first Zagreb index ${}^{m}M_1(G)$.
	\begin{cor}\label{pm4}
		Let $ G $ be a simple graph with $ n\geq3$ vertices and $ m $ edges. Let $ \Delta =d_1 \geq d_2 \geq \cdots \geq d_{n-1}\geq d_n=\delta>0$ be degree sequence. Then for any distinct indices $j \ne k \in \{1, 2, \dots, n\}$,
		\begin{equation*}\label{mp4}
			{}^{m}M_1(G) \geq \frac{ID(G)^2}{n}
			+ \frac{1}{2} \left( \frac{1}{d_j} - \frac{1}{d_k} \right)^2
			+ \frac{2n}{n - 2} \left( \frac{ID(G)}{n} - \frac{1}{d_j} + \frac{1}{d_k} \right)^2.
		\end{equation*}
	\end{cor}
	\begin{proof}
		For $\alpha = -1 $ and $\alpha =-2$, we find that
		\begin{equation}\label{me3}
			Z_{-1}(G) = \sum_{i=1}^n d_i^{-1} = \sum_{i=1}^n \frac{1}{d_i} = ID(G),
		\quad \text{and} \quad
			Z_{-2}(G) = \sum_{i=1}^n d_i^{-2} = \sum_{i=1}^n \frac{1}{d_i^2} = m_{M_1}(G).
		\end{equation}
		Hence, using \eqref{me3}, the Corollary \ref{mp4} follows immediately from Theorem \ref{zt1}.
	\end{proof}
	By selecting the extreme degrees $\Delta$ and $\delta$ in Corollary \ref{pm4},
	we obtain the following refinement of \eqref{eq:mZagrebBound}.
	\begin{cor}
		Let $ G $ be a simple graph with $ n\geq3$ vertices and $ m $ edges. Let $\Delta $ and $\delta > 0$ be the largest and smallest degree, respectively. Then
		\begin{equation}
			{}^{m}M_1(G) \geq\  \frac{ID(G)^2}{n}
			+ \frac{1}{2} \left( \frac{1}{\Delta} - \frac{1}{\delta} \right)^2
			+ \frac{2n}{n - 2} \left( \frac{ID(G)}{n} - \frac{1}{\Delta} + \frac{1}{\delta} \right)^2. \label{eq:first_bound_remark1} \\
		\end{equation}
		Equality holds in \eqref{eq:first_bound_remark1} if and only if $G$ is regular or $G \in \Gamma_{2,n} $ or $G \in \Gamma_{1,n-1} $ or $G \in \Gamma_{2,n-1} .$
	\end{cor}
	\begin{proof}
		The inequality \eqref{eq:first_bound_remark1} follows from Corollary \ref{pm4} by taking $d_j=\Delta$ and $d_k=\delta$.\\
	\end{proof}Note that the inequality \eqref{eq:first_bound_remark1} always provides better estimate than \eqref{eq:mZagrebBound}.\\
	The following corollary is an immediate consequence of Theorem \ref{zt2} for $\alpha=-1$.
	\begin{cor}\label{pm6}
		Let $G$ be  a simple graph with $n \geq 4$ vertices and $m$ edges.  Let $ \Delta =d_1 \geq d_2 \geq \cdots \geq d_{n-1}\geq d_n=\delta>0$ be the degree sequence. Then
		\begin{equation}
			\hspace{-.8cm} {}^{m}M_1(G) \geq\  \frac{1}{\Delta^2}
			+  \frac{\left(ID(G) - \frac{1}{\Delta} \right)^2}{n - 1}  
			+ \frac{1}{2} \left( \frac{1}{d_{2}} - \frac{1}{\delta} \right)^2 + \frac{2(n - 1)}{n - 3} \left( \frac{ID(G) - \frac{1}{\Delta}}{n - 1}
			- \frac{\frac{1}{d_{2}} - \frac{1}{\delta}}{2} \right)^2, \label{eq:second_bound_remark}
		\end{equation}Equality holds in \eqref{eq:second_bound_remark}  if and only if $G$ is regular, or $G \in \Gamma_{2,n-2}$, or $G \in \Gamma_{2,n-1}$, or $G \in \Gamma_{2,n}$, or $G \in \Gamma_{1,n-2}$, or $G \in \Gamma_{1,n-1}$.
	\end{cor}
	The classical arithmetic-harmonic mean inequality \cite{mitrinovic1970analytic} states that for any
	$n$ positive real numbers $a_1,a_2,\dots,a_n$,
	\begin{equation}
		\label{eq:am_hm}
		\left( \sum_{i=1}^n a_i \right) \left( \sum_{i=1}^n \frac{1}{a_i} \right) \ge n^2.
	\end{equation}
	Applying \eqref{eq:am_hm} to the $(n-1)$ positive real numbers
	$d_2,d_3,\dots,d_n$ gives
	\begin{equation}
		\label{eq:am_hm_d2_dn}
		\left( \sum_{i=2}^n d_i \right)\left( \sum_{i=2}^n \frac{1}{d_i} \right) \ge (n-1)^2.
	\end{equation}
	Since $\sum_{i=2}^n d_i = 2m - \Delta$ and
	$\sum_{i=2}^n \frac{1}{d_i} = ID(G) - \frac{1}{\Delta}$,
	therefore from \eqref{eq:am_hm_d2_dn}, we get
	\begin{equation}
		\label{eq:lower_bound_S_delta}
		ID(G) - \frac{1}{\Delta} \ge \frac{(n-1)^2}{2m - \Delta}.
	\end{equation}Therefore, inequality \eqref{eq:second_bound_remark} provides a stronger bound
	than \eqref{eq:ineq_upperbound1}.\\
	As a special case of Theorem~\ref{zt3} with $\alpha = -1$, we obtain the following corollary.
	\begin{cor}\label{pm7}
		Let $G$ be  a simple graph with $n \geq 5$ vertices and $m$ edges.  Let $ \Delta =d_1 \geq d_2 \geq \cdots \geq d_{n-1}\geq d_n=\delta>0$ be the degree sequence. Then
		\begin{equation}
			\small{}^{m}M_1(G) \geq\  \frac{1}{\Delta^2} + \frac{1}{\delta^2}
			+ \frac{\left(ID(G) - \frac{1}{\Delta} - \frac{1}{\delta} \right)^2}{(n - 2)}
			+ \frac{1}{2} \left( \frac{1}{d_{2}} - \frac{1}{d_{n-1}} \right)^2  + \frac{2(n - 2)}{n - 4} \left(
			\frac{ID(G) - \frac{1}{\Delta} - \frac{1}{\delta}}{n - 2}
			- \frac{ \frac{1}{d_{2}} + \frac{1}{d_{n-1}} }{2}
			\right)^2. \label{eq:zagreb_sharp_bound_remark}
		\end{equation}
		Equality holds in \eqref{eq:zagreb_sharp_bound_remark} if and only if $G$ is regular, or $G \in \Gamma_{3,n-2}$, or $G \in \Gamma_{2,n-2}$, or $G \in \Gamma_{3,n-1}$, or $G \in \Gamma_{2,n-1}$, or $G \in \Gamma_{3,n}$, or $G \in \Gamma_{2,n}$, or $G \in \Gamma_{1,n-2}$, or $G \in \Gamma_{1,n-1}$.
	\end{cor}
	\begin{remark}
		Note that, the inequality \eqref{eq:zagreb_sharp_bound_remark} is stronger than \eqref{eq:inverse_degree_bound}. To show this, it is suffice to show that
		\begin{equation}
			\frac{\left(ID(G) - \frac{1}{\Delta} - \frac{1}{\delta} \right)^2}{(n - 2)}  
			\geq
			\sqrt{ \frac{ \left( ID(G) - \frac{1}{\Delta} - \frac{1}{\delta} \right)^3 }{ 2m - \Delta - \delta } }.
			\label{eq:compare_step1}
		\end{equation}
		Squaring both sides of \eqref{eq:compare_step1} and simplifying
		\begin{equation*}
			ID(G) - \frac{1}{\Delta} - \frac{1}{\delta}
			\geq
			\frac{(n - 2)^2}{2m - \Delta - \delta},
			\label{eq:final_inequality1}
		\end{equation*}
		which is true by applying \eqref{eq:am_hm} to the $(n-2)$ positive real numbers $d_3,\dots,d_{n-1},$ and therefore \eqref{eq:zagreb_sharp_bound_remark} is also stronger than  \eqref{cor:delta-delta-bound}.
	\end{remark}
	
	\subsection*{Acknowledgment}
	The research of the first author is supported by the University Grants Commission (UGC), Government of India. The second author gratefully acknowledges the support received from the National Board for Higher Mathematics (NBHM), Department of Atomic Energy (DAE), Government of India (No. 02011/30/2025/NBHM(R.P)R\&D-II/9676).

	\bibliographystyle{plain}

\begin{thebibliography}{99}
		
		\bibitem{ahmad2025sombor}
		S. Ahmad, K.C. Das, A complete solution for maximizing the general Sombor index of chemical trees with given number of pendant vertices, Appl. Math. Comput. \textbf{505} (2025) 129532.
		
		\bibitem{ashrafi2010zagreb}
		A.R. Ashrafi, T. Došlić, A. Hamzeh, The Zagreb coindices of graph operations, Discrete Appl. Math. \textbf{158} (2010) 1571--1578.
		
		\bibitem{borovicanin2016extremal}
		B. Borovićanin, B. Furtula, On extremal Zagreb indices of trees with given domination number, Appl. Math. Comput. \textbf{279} (2016) 208--218.
		
		\bibitem{de1998upper}
		D. de Caen, An upper bound on the sum of squares of degrees in a graph, Discrete Math. \textbf{185} (1998) 245--248.
		
		
		\bibitem{das2003sharp}
		K.C. Das, Sharp bounds for the sum of the squares of the degrees of a graph, Kragujevac J. Math. \textbf{25} (2003) 31--41.
		
		\bibitem{das2013zagreb}
		K.C. Das, On Zagreb and Harary indices, MATCH Commun. Math. Comput. Chem. \textbf{70} (2013) 301--314.
		
		\bibitem{modzagreb}
		K.C. Das, N. Trinajstić, Comparison between first geometric-arithmetic index and modified first Zagreb index, MATCH Commun. Math. Comput. Chem. \textbf{70} (2013) 545--552.
		
		\bibitem{das2015zagreb}
		K.C. Das, K. Xu, J. Nam, Zagreb indices of graphs, Front. Math. China \textbf{10} (2015) 567--582.
		
		\bibitem{das2016comparison}
		K.C. Das, M. Dehmer, Comparison between the zeroth-order Randić index and the sum-connectivity index, Appl. Math. Comput. \textbf{274} (2016) 585--589.
		\bibitem{de2016f}
		N. De, S.U.A. Nayeem, A. Pal, The F-coindex of some graph operations, Springer Plus \textbf{5} (2016) $\#$221.
		
		\bibitem{doslic2008vertex}
		T. Došlić, Vertex-weighted Wiener polynomials for composite graphs, Ars Math. Contemp. \textbf{1} (2008) 66--80.
		
		
		\bibitem{furtula2015forgotten}
		B. Furtula, I. Gutman, A forgotten topological index, J. Math. Chem. \textbf{53} (2015) 1184--1190.
		\bibitem{gutman1972}
		I. Gutman, N. Trinajstić, Graph theory and molecular orbitals. Total $\pi$-electron energy of alternant hydrocarbons, Chem. Phys. Lett. \textbf{17} (1972) 535--538.
		
		\bibitem{gutman1975graph}
		I. Gutman, B. Ruščić, N. Trinajstić, C.F. Wilcox, Graph theory and molecular orbitals. XII. Acyclic polyenes, J. Chem. Phys. \textbf{62} (1975) 3399--3405.
		
		\bibitem{gutman2015zagreb}
		I. Gutman, B. Furtula, Ž. Kovijanić Vukićević, G. Popivoda, On Zagreb indices and coindices, MATCH Commun. Math. Comput. Chem. \textbf{74} (2015) 5--16.
		
		\bibitem{gutman2017generalizations}
		I. Gutman, K.C. Das, B. Furtula, E. Milovanović, I. Milovanović, Generalizations of Szőkefalvi Nagy and Chebyshev inequalities with applications in spectral graph theory, Appl. Math. Comput. \textbf{313} (2017) 235--244.
		
		\bibitem{hofmeister1988spectral}
		M. Hofmeister, Spectral radius and degree sequence, Math. Nachr. \textbf{139} (1988) 37--44.
		
		\bibitem{hu2005molecular}
		Y. Hu, X. Li, Y. Shi, T. Xu, I. Gutman, On molecular graphs with smallest and greatest zeroth-order general Randić index, MATCH Commun. Math. Comput. Chem. \textbf{54} (2005) 425--434.
		%
		
		\bibitem{li2019second}
		J. Li, J. Zhang, On the second Zagreb eccentricity indices of graphs, Appl. Math. Comput. \textbf{352} (2019) 180--187.
		
		\bibitem{mansour2012and}
		T. Mansour, C. Song, The $a$ and $(a,b)$-analogs of Zagreb indices and coindices of graphs, Int. J. Combin. \textbf{2012} (2012) Article ID 909285.
		
		\bibitem{mansour2016new}
		T. Mansour, M.A. Rostami, E. Suresh, G.B.A. Xavier, New sharp lower bounds for the first Zagreb index, Sci. Publ. State Univ. Novi Pazar Ser. A \textbf{8} (2016) 11--19.
		\bibitem{matejic2022bounds}
		M. Matejić, Ş.B.B. Altındağ, E. Milovanović, I. Milovanović, On the bounds of zeroth-order Randić index, Filomat \textbf{36} (2022) 6443--6456.
		
		
		\bibitem{milicevic2004variable}
		A. Milićević, S. Nikolić, On variable Zagreb indices, Croat. Chem. Acta \textbf{77} (2004) 97--101.
		\bibitem{milovanovic2015sharp}
		E.I. Milovanović, I.Ž. Milovanović, Sharp bounds for the first Zagreb index and
		first Zagreb coindex, Miskolc Math. Notes \textbf{16} (2015) 1017--1024.
		
		\bibitem{milovsevic2019some}
		P. Milošević, I. Milovanović, E. Milovanović, M. Matejić, Some inequalities for general zeroth-order Randić index, Filomat \textbf{33} (2019) 5249--5258.
		
		\bibitem{milovanovic2020note}
		I. Milovanović, M. Matejić, E. Milovanović, A note on the general zeroth-order Randić coindex of graphs, Contrib. Math. \textbf{1} (2020) 17--21.
		
		\bibitem{milovanovic2023some}
		I. Milovanović, E. Milovanović, M. Matejić, A. Ali, Some new bounds on the modified first Zagreb index, Commun. Comb. Optim. \textbf{8} (2023) 13--21.
		\bibitem{mitrinovic1970analytic}
		D.S. Mitrinović, P.M. Vasić, Analytic Inequalities, Springer-Verlag, New York, 1970.
		\bibitem{mukwembi2010diameter}
		S. Mukwembi, On diameter and inverse degree of a graph, Discrete Math. \textbf{310} (2010) 940--946.
		
		
		\bibitem{yang2022maximum}
		J. Yang, H. Deng, Maximum first Zagreb index of orientations of unicyclic graphs with given matching number, Appl. Math. Comput. \textbf{427} (2022) 127131.
		
		\bibitem{yoon2006relationship}
		Y.S. Yoon, J.K. Kim, A relationship between bounds on the sum of squares of degrees of a graph, J. Appl. Math. Comput. \textbf{21} (2006) 233--238.
		
		\bibitem{yu2004spectral}
		A.M. Yu, M. Liu, F. Tian, On the spectral radius of graphs, Linear Algebra Appl. \textbf{387} (2004) 41--49.
		
		
	\end{thebibliography}

\end{document}